\documentclass{amsart}
\numberwithin{equation}{section}
\usepackage[english]{babel}
\usepackage[T1]{fontenc}
\usepackage[latin1]{inputenc}
\usepackage{indentfirst}
\usepackage{enumitem}
\usepackage{amsmath,amssymb, amsbsy}
\usepackage{amsfonts}
\usepackage{hyperref}
\usepackage{cleveref}
\usepackage{esint}
\usepackage{latexsym}
\usepackage{amsthm}
\usepackage[dvips]{graphicx}
\usepackage{xcolor}
\usepackage{tikz}
\usepackage{pgfplots}
\usetikzlibrary{intersections}
\usepackage{tikz-3dplot}
\usepackage[outline]{contour}
\DeclareGraphicsExtensions{.pdf,.png,.jpg,.eps}

\usepackage{comment}

\usepackage{amsthm}

\usepackage{color}
\usepackage[latin1]{inputenc}
\usepackage[active]{srcltx}
\definecolor{Arancio}{cmyk}{0,0.61,0.87,0}

\definecolor{blus}{RGB}{0,102,204}

\newcommand{\brd}[1]{\mathbb{#1}}
\newcommand{\R}{\brd{R}}

\newcommand{\be}{\begin{equation}}
\newcommand{\ee}{\end{equation}}
\newcommand{\loc}{{\text{\tiny{loc}}}}

\newtheorem{teo}{Theorem}[section]

\newtheorem{Lemma}[teo]{Lemma}
\newtheorem{Theorem}[teo]{Theorem}

\theoremstyle{definition}

\newtheorem{remark}[teo]{Remark}

\newtheorem*{problem}{Problem}
\newtheorem{conjecture}[teo]{Conjecture}

\usepackage{physics}
\usepackage{amsmath}
\usepackage{tikz}
\usepackage{mathdots}
\usepackage{yhmath}
\usepackage{cancel}
\usepackage{color}
\usepackage{siunitx}
\usepackage{array}
\usepackage{multirow}
\usepackage{amssymb}
\usepackage{gensymb}
\usepackage{tabularx}
\usepackage{extarrows}
\usepackage{booktabs}
\usetikzlibrary{fadings}
\usetikzlibrary{patterns}
\usetikzlibrary{shadows.blur}
\usetikzlibrary{shapes}

\begin{document}


\title[Boundary unique continuation in planar domains by conformal mapping]
{Boundary unique continuation in planar domains\\ by conformal mapping}
\date{\today}

\author{Stefano Vita}

\address[S. Vita]{Dipartimento di Matematica ``F. Casorati"
\newline\indent
Universit\`a di Pavia
\newline\indent
Via Ferrata 5, 27100, Pavia, Italy}
\email{stefano.vita@unipv.it}

\maketitle

\begin{abstract}
Let $\Omega\subset\R^2$ be a chord arc domain. We give a simple proof of the the following fact, which is commonly known to be true: a nontrivial harmonic function which vanishes continuously on a relatively open set of the boundary cannot have the norm of the gradient which vanishes on a subset of positive surface measure (arc length). This result is conjectured to be true in higher dimensions by Lin, in Lipschitz domains. Let now $\Omega\subset\R^2$ be a $C^1$ domain with Dini mean oscillations. We prove that a nontrivial harmonic function which vanishes continuously on a relatively open subset of the boundary $\partial\Omega\cap B_1$ has a finite number of critical points in $\overline\Omega\cap B_{1/2}$. The latter improves some recent results by Kenig and Zhao. Our technique involves a conformal mapping which moves the boundary where the harmonic function vanishes into an interior nodal line of a new harmonic function, after a further reflection. Then, size estimates of the critical set - up to the boundary - of the original harmonic function can be understood in terms of estimates of the \emph{interior} critical set of the new harmonic function and of the critical set - up to the boundary - of the conformal mapping.
\end{abstract}

\section{Introduction}
Let $\Omega$ be a domain in the plane $\R^2$. In this paper we are concerned with the local behaviour of planar harmonic functions in $\Omega$ near a given open piece of the boundary where they vanish continuously; that is, we consider weak solutions to
\begin{equation}\label{eq}
\begin{cases}
\Delta u=0 &\mathrm{in} \ \Omega\cap B_1,\\
u=0 &\mathrm{on} \ \partial\Omega\cap B_1.
\end{cases}
\end{equation}
Here we assume that $0\in\partial\Omega$, $z=(x,y)\in\R^2$, $B_1=\{|z|<1\}$ and $\Omega\cap B_1$ is simply connected. The aim of our work is to provide a conformal mapping $\Theta$ which locally moves the above boundary problem into 
\begin{equation*}\label{eq2}
\begin{cases}
\Delta U=0 &\mathrm{in} \ \{Y>0\}\cap B_1,\\
U=0 &\mathrm{on} \ \{Y=0\}\cap B_1.
\end{cases}
\end{equation*}
Here $u=U\circ\Theta$ and $Z=(X,Y)=\Theta(x,y)$. After an odd reflection of $U$ across the line $\{Y=0\}$ (Schwarz reflection principle), one ends up with a harmonic function in a ball. At the end of this procedure, the map $\Theta$ turned the original boundary $\partial\Omega$ into an interior nodal line of $U$
$$\Theta(\partial\Omega\cap B_1)\subset\{Y=0\}.$$
Then, the boundary behaviour of the solution $u$ can be understood locally in terms of the \emph{interior} behaviour of the new harmonic function $U$ and the boundary behaviour of the conformal mapping $\Theta$. This transformation provides solutions to classical problems in boundary unique continuation.
\subsection{Boundary unique continuation in chord arc domains}
\emph{Boundary unique continuation} typically concerns the following question: given a nontrivial solution $u$ of \eqref{eq}, and depending on the regularity of the domain $\Omega\subset\R^n$ with $n\geq2$, how big can the singular set $S(u)=\{u=|\nabla u|=0\}$ be - whenever it makes sense - restricted to the boundary? Similar questions can be raised for the full critical set $C(u)=\{|\nabla u|=0\}$ inside the domain and up to the boundary. In this context, a famous problem was proposed by Lin \cite{Lin91}.

\begin{conjecture}\label{conjecture}
Let us consider a harmonic function $u$ in a Lipschitz domain $\Omega$ in $\R^n$ vanishing continuously on a relatively open subset $V$ of the boundary $\partial\Omega$. Suppose that the normal derivative $\partial_\nu u$ vanishes in a subset of $V$ with positive surface measure. Then $u\equiv0$.
\end{conjecture}
The validity of the result above was first established in $C^{1,1}$ domains \cite{Lin91}, convex Lipschitz domains \cite{AdoEscKen95}, $C^{1,\alpha}$ and $C^{1,\mathrm{Dini}}$ domains \cite{AdoEsc97,KukNys98}, $C^1$ domains and Lipschitz domains with small Lipschitz constant \cite{Tol23}, quasiconvex Lipschitz domains \cite{Cai24}. We would like to mention also \cite{Mcc23,Gal23} for estimates of singular and critical sets in case of Lipschitz convex domains and Lipschitz domains with small Lipschitz constant, respectively. The conjecture in its generality is still open.\\

The two dimensional case is commonly known to be true, even without requiring that $u$ vanishes on the given open piece of the boundary and in general simply connected domains with rectifiable boundary. The known argument uses the F. and M. Riesz theorem, applied to the conformal mapping from the disc, and the subharmonicity of $\log|\nabla u|$, see discussions in \cite{AdoEsc97,Wol91}.

Among the other motivations, the paper aims to give an alternative proof of Conjecture \ref{conjecture} in the two dimensional case and when $\Omega$ is a chord arc domain. In two dimensions, a bounded chord arc domain is a Jordan domain for which the boundary is locally rectifiable and there exists a constant $\lambda>0$ such that
$$\sigma(\gamma(z_1,z_2))\leq\lambda|z_1-z_2|,\qquad\forall z_1,z_2\in\partial\Omega,$$
where $\gamma(z_1,z_2)$ is the shortest arc in the boundary connecting $z_1$ and $z_2$, and $\sigma(\gamma(z_1,z_2))$ is its length. In general, a chord arc domain in $\R^n$ is a non-tangentially accessible (NTA) domain \cite{JerKen82} whose boundary is Ahlfors-David regular, i.e. the surface measure on boundary balls of radius $r$ grows like $r^{n-1}$. Our first result can be stated as follows

\begin{Theorem}\label{teo1}
Let us consider a harmonic function in a chord arc domain $\Omega$ in $\R^2$ vanishing continuously on a relatively open subset $V$ of the boundary $\partial\Omega$. Suppose that the norm of the gradient $|\nabla u|$ vanishes in a subset of $V$ with positive arc length. Then $u\equiv0$. Actually, given a chord arc domain $\Omega\subset\R^2$ and a nontrivial solution $u$ to \eqref{eq}, then
\begin{equation*}\label{Sn-1}
\mathcal H^1(C(u)\cap \overline\Omega\cap B_{1})=0.
\end{equation*}
\end{Theorem}

The idea of the proof we propose here is the following: after composing with the conformal mapping $\Theta$, one has
\begin{equation*}\label{count}
|\nabla u|^2=|\mathrm{det} D\Theta| \cdot |\nabla U|^2\circ\Theta,
\end{equation*}
where $D\Theta$ stands for the Jacobian of $\Theta$. Then, the critical set of $u$ is locally controlled in size by the critical set of $U$ and by the critical set of the conformal mapping
\begin{equation*}\label{CriticalTheta}
C(\Theta)=\{|\mathrm{det} D\Theta|=0\}.
\end{equation*}
Roughly speaking, the first is \emph{small}, being an \emph{interior} critical set - up to perform an odd reflection - of a harmonic function (it consists of a finite number of points), and the latter is concentrated along the boundary $\partial\Omega$ and has zero surface measure. The critical set of the conformal mapping we provide is in fact the critical set of an auxiliary positive harmonic function vanishing on $\partial\Omega$; that is, 
\begin{equation*}
|\mathrm{det} D\Theta|=|\nabla v|^2,\qquad \begin{cases}
\Delta v=0 &\mathrm{in} \ \Omega\cap B_1,\\
v>0 &\mathrm{in} \ \Omega\cap B_1,\\
v=0 &\mathrm{on} \ \partial\Omega\cap B_1.
\end{cases}
\end{equation*}

The function above can be easily constructed without interior critical points, following ideas in \cite{Ale86}, and this is crucial for having local invertibility of the conformal mapping.
Moreover, along the boundary, $|\nabla v|$ is $\sigma$-a.e. comparable to the density of the harmonic measure with respect to the surface measure $d\sigma$. Besides, the harmonic measure is mutually absolutely continuous with respect to $d\sigma$ \cite{Lav36}. The latter two facts imply that the critical set of $v$ along the boundary has zero surface measure. In the two dimensional case, by surface measure we mean the arc length, which corresponds to the one dimensional Hausdorff measure restricted to $\partial\Omega$, i.e. $\sigma=\mathcal H^1\llcorner\partial\Omega$.

\subsection{$(n-2)$ dimensional size control of singular and critical sets in $C^{1,\mathrm{DMO}}$ domains}

When the boundary is more regular, one may ask for stronger information on the size of singular and critical sets. That is, consider the following

\begin{problem}\label{p3}
Let $u$ be a nontrivial solution to \eqref{eq} in a domain $\Omega\subset\R^n$ with $n\geq2$. Identify the conditions on the boundary under which one has
\begin{equation}\label{Sn-2}
\mathcal H^{n-2}(S(u)\cap \overline\Omega\cap B_{1/2} )\leq C,
\end{equation}
or either
\begin{equation}\label{Cn-2}
\mathcal H^{n-2}(C(u)\cap \overline\Omega\cap B_{1/2})\leq C.
\end{equation}
\end{problem}

This problem was recently addressed by Kenig and Zhao in a series of papers \cite{KenZha22a,KenZha23,KenZha24}, using also techniques developed by Naber and Valtorta \cite{NabVal17b}. The property \eqref{Sn-2} holds true in $C^{1,\mathrm{Dini}}$ domains \cite{KenZha22a}, and counterexamples are provided below this threshold \cite{KenZha23}. Actually, in \cite{KenZha22a} the authors obtain upper estimates for the $(n-2)$ dimensional Minkowski content, which coincides with the $(n-2)$ dimensional Hausdorff measure in the present planar case. The property \eqref{Cn-2}, which is stronger than \eqref{Sn-2}, holds in $C^{1,\alpha}$ domains \cite{KenZha24}. The size bounds in \cite{KenZha22a,KenZha24} are uniform prescribing a bound on the Almgren frequency function of the solution at a macroscopic scale as well as a control over the $C^{1,\mathrm{Dini}}$ (respectively $C^{1,\alpha}$) character of the boundary parametrization.\\

Our conformal approach allows us to prove the stronger property \eqref{Cn-2}, again in two dimensions, any time the critical set of the conformal mapping has the suitable bound in measure
\begin{equation*}
\mathcal H^{0}(C(\Theta)\cap \overline\Omega\cap B_{1/2})\leq C.
\end{equation*}
Here $\mathcal H^0$ stands for the counting measure. Our focus here is not on establishing conditions on the boundary that yield the precise bound above, although many examples could be constructed. Instead, we concentrate our analysis on a set of hypothesis which implies an empty critical set, i.e.
\begin{equation*}
C(\Theta)\cap \overline\Omega\cap B_{1/2}=\emptyset.
\end{equation*}
With respect to our conformal mapping, this happens any time the following two conditions hold:
\begin{itemize}
\item[(P1)] any solution $u$ to \eqref{eq} belongs to $C^1_\loc(\overline\Omega\cap B_1)$ (and the same holds true for solutions having homogeneous Neumann boundary condition at $\partial\Omega$);
\item[(P2)] the Hopf lemma holds true: any solution $v$ to \eqref{eq} which is positive in $\Omega\cap B_1$ has $\partial_\nu v<0$ on $\partial\Omega\cap B_1$.
\end{itemize}

As we proved in \cite{DonJeoVit23}, the properties above hold true when the domain is $C^1$ with Dini mean oscillations ($C^{1,\mathrm{DMO}}$), and this is true in any dimension. This class of domains was recently introduced in \cite{DonJeoVit23} and strictly contains the $C^{1,\mathrm{Dini}}$ class, see Section \ref{s:DMO} for the precise definition. For reader's convenience, we would like to provide here a two dimensional example of a $C^{1,\mathrm{DMO}}$ local parametrization which fails to be $C^{1,\mathrm{Dini}}$: the domain is given, locally around $0\in\partial\Omega$, by
\begin{equation*}
\Omega\cap B_{1/2}=\{y>\varphi(x)\}\cap B_{1/2},\qquad \partial\Omega\cap B_{1/2}=\{y=\varphi(x)\}\cap B_{1/2},
\end{equation*}
with $z=(x,y)\in\R^{2}$ and
\begin{equation*}
\varphi(x)=\frac{x}{|\log |x||^{1/2}},\qquad |x|<1/2.
\end{equation*}
The following is our second result

\begin{Theorem}\label{t:DMO}
Let $n=2$, $\Omega$ be a $C^{1,\mathrm{DMO}}$ domain and $u$ be a nontrivial solution to \eqref{eq}. Then,
\begin{equation*}\label{B01}
\mathcal H^{n-2}(C(u)\cap \overline\Omega\cap B_{1/2})<\infty.
\end{equation*}
\end{Theorem}

We believe that the above result is still valid in any dimension $n\geq2$. Let us stress the fact that Theorem \ref{t:DMO} is not in contradiction with \cite{KenZha23}, since the counterexamples proposed there \emph{do not see} our intermediate condition. We refer to Remark \ref{r:DMO} for a detailed explanation of this fact.\\

Finally, we would like to emphasize that hodograph conformal mappings, like the one we introduce in Section \ref{s:hodograph}, have been previously employed in other contexts such as the structure and regularity of nodal sets of harmonic functions, univalent $\sigma$-harmonic mappings, two-phase free boundary problems, and boundary Harnack principles on nodal domains, see respectively \cite{HarWin53,AleNes01,DePSpoVel24,TerTorVit24}.

\section{The construction of the conformal mapping in chord arc domains}\label{s:CAD}
The conformal mapping $\Theta$ we are going to construct is the same for proving both Theorem \ref{teo1} and Theorem \ref{t:DMO}. However, the construction in the first case is more delicate, since in chord arc domains one has to work with generalized gradients, defined as non-tangential limits, and the invertibility of the map is more subtle. The full section should be intended as the proof of Theorem \ref{teo1}.

\subsection{Chord arc domains}
In this brief subsection we would like to introduce the chord arc domains.
A NTA (non-tangentially accessible) domain $\Omega$ is one which enjoys an interior Harnack Chain condition, as well as interior and exterior Corkscrew conditions. This notion was introduced in \cite{JerKen82}. A chord arc domain is a NTA domain whose boundary is Ahlfors-David regular, i.e. the surface measure on boundary balls of radius $r$ grows like $r^{n-1}$. We refer to \cite{AzzHofMarNysTor17} for precise definitions and nice characterizations of chord arc domains. We also would like to refer to \cite{Lav36} for the proof of the mutual absolute continuity between arc length and harmonic measure in chord arc domains, see also \cite{Dah86,DavJer90,Sem90} for the result in any dimension and \cite{CapKenLan05,JerKen82b} for further references.
Finally, we refer to \cite{LewNys12} for useful considerations on the gradient of positive harmonic functions vanishing continuously on a relatively open subset of a chord arc boundary.
\subsection{Non-tangential limits}
Given a chord arc bounded domain $\Omega$ in $\R^n$ and the surface measure $\sigma$ on $\partial\Omega$, one has for $\sigma$-a.e. $z$ the existence of the tangent plane to $\partial\Omega$ in $z$ and of the outer unit normal vector $\nu(z)$.
Then, given any parameter $\alpha>0$, let us consider the non-tangential approach region to a point on the boundary $z\in\partial\Omega$

\begin{equation*}
\Gamma_\alpha(z)=\{\xi\in\Omega\, : \, (1+\alpha)d(\xi,\partial\Omega)>|z-\xi|\},
\end{equation*}
and, given a measurable function $w$ defined in $\Omega$, the non-tangential maximal function at the boundary point $z\in\partial\Omega$
$$\mathcal N_\alpha w(z)=\sup_{\xi\in\Gamma_\alpha(z)} |w|(\xi).$$
The choice of $\alpha$ does not play any role later on, so we can fix $\alpha=1$ and simply write $\Gamma(z)=\Gamma_1(z)$ and $\mathcal Nw(z)=\mathcal N_1w(z)$. Then, we say that $w$ converges non-tangentially to $f$ at $z\in\partial\Omega$ if
\begin{equation*}
\lim_{\xi\in \Gamma(z), \ \xi\to z}w(\xi)=f(z).
\end{equation*}

\subsection{Positive harmonic function vanishing on $\partial\Omega$ with no interior critical points}\label{s:distance}
The first step is the construction of a solution to
\begin{equation}\label{eqv}
\begin{cases}
\Delta v=0 & \mathrm{in} \ \Omega\cap B_1\\
v>0 & \mathrm{in} \ \Omega\cap B_1\\
v=0 & \mathrm{on} \ \partial\Omega\cap B_1,
\end{cases}
\end{equation}
with no interior critical points and simple level curves with endpoints on the boundary.

\begin{Lemma}\label{lem:1}
There exists a solution to \eqref{eqv}, positive and with no critical points in $\Omega\cap B_1$. Moreover, for every $\ell>0$, the level set $\{v=\ell\}\cap (\Omega\cap B_1)$ is either empty or consists of a single simple curve whose endpoints lie on $\partial(\Omega\cap B_1)$.
\end{Lemma}

\begin{proof}
The existence of such a function is given by the solvability of the Dirichlet problem for the Laplacian on the simply connected chord arc domain $\Omega\cap B_1$, which is regular for the solvability of the Dirichlet problem. Let us prescribe as Dirichlet data on $\partial(\Omega\cap B_1)$ a continuous nonnegative function which is positive and unimodal (with a unique maximum) on a relatively open simple arc $\gamma$ in $\partial B_1\cap\Omega$, and zero elsewhere. In other words, the unimodal prescribed function on $\gamma$ is monotone increasing in a first arc $\gamma_1$ and then decreasing in a second one $\gamma_2$ with $\gamma=\gamma_1\cup\gamma_2$ and $\gamma_1\cap\gamma_2=\{z_0\}$. Then, $v$ is continuous in $\overline{\Omega\cap B_1}$ and by the strong maximum principle $v$ is positive inside $\Omega\cap B_1$. Moreover, $v$ has no interior critical points in $\Omega\cap B_1$ (see e.g. \cite[Theorem 1.2]{Ale86} assuming the number of maxima is $N=1$). Moreover, since $v$ has no interior critical points, each level set $\{v=\ell\}$ is either empty in $\Omega\cap B_1$ or a union of regular curves. Closed components are excluded by the maximum principle. Hence, each connected component of $\{v=\ell\}$ must intersect the boundary. Finally, the unimodality of the boundary data implies that, for each $\ell>0$, if $\{v=\ell\}\cap (\Omega\cap B_1)\neq\emptyset$, there are exactly two boundary points where $v=\ell$, and therefore $\{v=\ell\}$ consists of a single simple curve connecting them.
\end{proof}

\subsection{The size of the singular set of positive harmonic functions at the chord arc boundary}
Now, being $v$ a nonnegative solution to \eqref{eqv} in a bounded chord arc domain, by \cite{Lav36} and \cite[Theorem 1]{LewNys12} we have that
\begin{itemize}
\item[(i)] the \emph{harmonic measure} $\omega$ is absolutely continuous with respect to $\sigma$ on $\partial\Omega\cap B_1$ and $d\omega\in A_\infty(\partial\Omega\cap B_1,d\sigma)$;
\item[(ii)] the limit
\begin{equation*}
\nabla v(z):=\lim_{\xi\in \Gamma(z), \ \xi\to z}\nabla v(\xi)
\end{equation*}
exists for $\sigma$-a.e. $z\in\partial\Omega\cap B_1$. Moreover, $\nabla v(z)=-|\nabla v(z)|\nu(z)$ where $\nu$ stands for the outer unit normal vector;
\item[(iii)] there exists $p>1$ such that $\mathcal N|\nabla v|\in L^p(\partial\Omega\cap B_1,d\sigma)$;
\item[(iv)] $d\omega=|\nabla v|d\sigma$ for $\sigma$-a.e. $z\in\partial\Omega\cap B_1$.
\end{itemize}

Summing up the information above, we have that $|\nabla v|$ can not vanish on a set of positive surface measure. Then, combining this information with the fact that $v$ has no interior critical points, we have
\begin{equation}\label{criticalv}
\mathcal H^1(C(v)\cap \overline\Omega\cap B_1)=0,
\end{equation}
with all the critical points (which are singular points) concentrated along the boundary.

\subsection{The harmonic conjugate}\label{s:conjugate}
Let us construct the harmonic conjugate of $v$ by solving
\begin{equation}\label{form}
\nabla\overline v=J\nabla v \quad\mathrm{in} \ \Omega\cap B_1,\qquad J=\begin{pmatrix}
    0     & 1 \\
    -1   &0
\end{pmatrix}.
\end{equation}
Here $J$ is the clockwise rotation matrix of angle $\pi/2$, such that $J^{-1}=J^T=-J$. The condition above corresponds to the Cauchy-Riemann equations. The form is closed in a simply connected domain, so \eqref{form} admits a solution. In particular $\overline v$ has the following properties
\begin{equation*}
\nabla v\cdot\nabla \overline v=0,\qquad |\nabla v|=|\nabla\overline v|,
\end{equation*}
and is harmonic in $\Omega\cap B_1$. Moreover, we can suppose that $\overline v(0)=0$, considering instead $\tilde v(z)=\overline v(z)-\overline v(0)$. Notice that also $\overline v$ admits a non-tangential extension on $\partial\Omega$ for $\sigma$-a.e. point on the boundary (always denoted by $\overline v$), and the same holds for its gradient, which belongs to $L^p(\partial\Omega\cap B_1,d\sigma)$.

\subsection{A hodograph conformal transformation and its invertibility}\label{s:hodograph}
Let us define the hodograph conformal mapping involving $v,\overline v$
\begin{equation*}\label{conformal}
\Theta(x,y)=(\overline v(x,y),v(x,y))=(X,Y),
\end{equation*}
with $\Theta(0)=0$.

\begin{Lemma}\label{lem:2}
The map $\Theta:\Omega\cap B_1\to\R^2$ is conformal.
\end{Lemma}

\begin{proof}
The conformality of $\Theta$ should be understood as the conformality of the complex map $f(z)=v(z)+i\overline v(z)$. The latter is verified if and only if $f$ is holomorphic and its complex derivative is everywhere non-zero on $\Omega\cap B_1$ as a subset of $\mathbb C$. The first condition is due to the real analyticity of $v,\overline v$ in $\Omega\cap B_1$ and the validity of the Cauchy-Riemann equations, implied by the definition of the harmonic conjugate. The second condition follows by the following remark: the complex derivative satisfies
\[
|f'(z)| = |\nabla v(z)|.
\]
By Lemma \ref{lem:1}, $v$ has no interior critical points in $\Omega\cap B_1$. Then, $|\nabla v(z)|\neq 0$ in $\Omega\cap B_1$.
\end{proof}

Now, we need to prove injectivity of the map in $\Omega\cap B_1$. The Jacobian associated with $\Theta$ is given by
\begin{align*}
D\Theta=\begin{pmatrix}
    \partial_x\overline v       & \partial_y\overline v \\
    \partial_x v      & \partial_y  v
\end{pmatrix}, \qquad \mathrm{with}\quad |\mathrm{det} \, D\Theta|=|\nabla v|^2=|\nabla \overline v|^2.
\end{align*}
Hence, the fact that $v$ has no interior critical points implies the local invertibility of the map in $\Omega\cap B_1$. Then, the global injectivity of $\Theta$ in $\Omega\cap B_1$ is proved in the following lemma. 

\begin{Lemma}\label{lem:3}
The map $\Theta$ is injective in $\Omega\cap B_1$.
\end{Lemma}

\begin{proof}
Let's consider two points $z_1\neq z_2$ in $\Omega\cap B_1$. If $v(z_1)\neq v(z_2)$ we have the desired condition $\Theta(z_1)\neq\Theta(z_2)$. Hence, we can suppose $v(z_1)= v(z_2)$; that is, the two points lie on the same level curve of $v$; that is, $z_1,z_2\in\{v=\ell\}\cap B_1$ for some $\ell>0$. By Lemma \ref{lem:1}, the level set $\{v=\ell\}$ is a simple curve. In particular, it is connected, and there exists a path $\gamma$ in $\{v=\ell\}$ connecting $z_1$ and $z_2$; that is, $\gamma:[t_1,t_2]\to\R^2$ with $\mathrm{supp}\gamma\subset\{v=\ell\}\cap B_1$ and $\gamma(t_1)=z_1$, $\gamma(t_2)=z_2$. The curve is analytic, being a path along the analytic level curve of $v$. We choose the orientation in such a way that a clockwise rotation of angle $\pi/2$ of the tangent vector to the curve $\gamma'(t)$ goes in the same direction of the outer unit normal vector $\nu(\gamma(t))$, i.e. outward with respect to the super-level set $\{v>\ell\}$. Then
\begin{equation*}\label{curveintegral}
\overline v(z_2)-\overline v(z_1)=\int_{t_1}^{t_2}J\nabla v(\gamma(t))\cdot\gamma'(t) \, dt>0.
\end{equation*}
Here, we are using the fact that the gradient of $v$ equals $-|\nabla v|\nu$ and it is nonzero. Notice that $-J\nu(\gamma(t))$ is parallel to $\gamma'(t)$ and goes in the same direction. This means that, restricted to the level curve $\{v=\ell\}$, the first component $\Theta_1(\gamma(t))=\overline v(\gamma(t))$ of the map $\Theta\circ\gamma$ is monotone increasing in $t$, and hence it is injective.
\end{proof}

Let us remark that Lemma \ref{lem:2} together with Lemma \ref{lem:3} imply the inveribility of $\Theta$. Actually, we can conclude that $\Theta$ is a biholomorphism between the open sets $\Omega\cap B_1$ and $\Theta(\Omega\cap B_1)$.

Finally, we want to further localize inside the original domain $\Omega\cap B_1$ in a way that will ensure nice properties for the image domain after composing with $\Theta$. Let's consider a fixed $1/2<r<1$ so that the intersection $\partial B_r\cap\partial\Omega$ consists of two points $\tilde z_1,\tilde z_2$ where both the tangent plane and the normal vector to $\partial\Omega$ are well defined, and $v$ has a well defined nontrivial non-tangential gradient. Notice that it is always possible to find such $r$, since the boundary set where the latter conditions are not satisfied is of zero surface measure. Then, we can localize in a subdomain $E_r$ which contains $\Omega\cap B_{1/2}$ and it is contained in $\Omega\cap B_{1}$. Such a domain is constructed as the portion of the plane between $\partial\Omega$ and a regular curve $\eta$ connecting $\tilde z_1,\tilde z_2$. The curve $\eta$ consists of two small segments near both $\tilde z_1,\tilde z_2$ and globally does not touch neither $\partial B_{1/2}\cap\Omega$, or $\partial B_{1}\cap\Omega$. The segments have directions given respectively by $\nabla v(\tilde z_i)=-|\nabla v(\tilde z_i)|\nu (\tilde z_i)$, for $i=1,2$. This property implies that $\Theta(E_r)$ has orthogonal (and yet vertical) crossing with $\{Y=0\}$ in the new coordinate system $(X,Y)$. Moreover, a portion of the boundary of $\Theta(E_r)$ is contained in $\{Y=0\}$ (one of the two paths connecting $\Theta(\tilde z_1)$ to $\Theta(\tilde z_2)$), i.e. $\Theta$ is locally straightening the boundary $\partial\Omega$.

\subsection{Size control of the critical set}
Let us consider $U=u\circ\Theta^{-1}$, which solves
\begin{equation*}
\begin{cases}
\Delta U=0 &\mathrm{in} \ \Theta(E_r)\\
U=0 &\mathrm{on} \ \partial\Theta(E_r)\cap\{Y=0\}.
\end{cases}
\end{equation*}
Hence, considering the odd reflection $U(X,Y)=-U(X,-Y)$ across $\{Y=0\}$, one ends up with a harmonic function on an open regular set $\mathcal U$ (the reflected $\Theta(E_r)$) for which $\Theta(\partial\Omega)$ is an interior nodal line $\{Y=0\}$. Then, since
\begin{equation*}
|\nabla u|^2=|\mathrm{det} D\Theta| \cdot |\nabla U|^2\circ\Theta,
\end{equation*}
we have
\begin{equation*}\label{Sn-1}
\mathcal H^1(C(u)\cap \overline{E_r})\leq \mathcal H^1(C(\Theta)\cap \overline{E_r})+\mathcal H^1(C(U)\cap \mathcal U)=0.
\end{equation*}
The latter is true due to \eqref{criticalv} and classic size estimates of interior critical sets of harmonic functions. In particular, the interior critical set of the reflected $U$ in $\mathcal U$ consists of a locally finite number of isolated points. The latter easily follows by seeing $U$ as the real part of a holomorphic function $f$ and characterizing its critical points in terms of zeroes of the complex derivative $f'$.

\section{The construction of the conformal mapping in $C^{1,\mathrm{DMO}}$ domains}
The conformal mapping $\Theta$ we consider for the proof of Theorem \ref{t:DMO} is the same we built in the previous section, but enjoys better properties. The full section should be intended as the proof of Theorem \ref{t:DMO}.

\subsection{$C^{1,\mathrm{DMO}}$ domains}\label{s:DMO}
First, let us recall the definition of $C^{1,\mathrm{Dini}}$ domains in $\R^n$ with $n\geq2$. In this case, the local boundary parametrization $\varphi$ is a $C^1$ function and the modulus of continuity of its gradient is a Dini function. This means that locally
\begin{equation}\label{localparame}
\Omega\cap B_1=\{x_n>\varphi(x')\}\cap B_1,\qquad \partial\Omega\cap B_1=\{x_n=\varphi(x')\}\cap B_1,
\end{equation}
with $x=(x',x_n)\in\R^{n-1}\times\R$, $\varphi\in C^1(\overline{B_1'})$ with $B_1'=B_1\cap\{x_n=0\}$, $\varphi(0)=0$, $\nabla_{x'}\varphi(0)=0$. Then, there exist a positive constant and a modulus of continuity $\eta$ such that for all $i=1,...,n-1$ and $x',y'\in B_1'$
$$|\partial_i\varphi(x')-\partial_i\varphi(y')|\leq C\eta(|x'-y'|)$$
with
\begin{equation}\label{DINI}
\int_0^1\frac{\eta(r)}{r}dr<\infty.
\end{equation}

Let us now proceed with the definition of $C^{1,\mathrm{DMO}}$ domains. Here, the boundary of the domain is locally parametrized by a $C^1$ function $\varphi$ whose partial derivatives $\partial_i\varphi$ are of Dini mean oscillations. In other words, the parametrization is as in \eqref{localparame} with $\varphi\in C^1(\overline{B_1'})$, $\varphi(0)=0$, $\nabla_{x'}\varphi(0)=0$ and

\begin{equation*}\label{DMO}
\eta_i(r)=\sup_{x_0\in B_1'}\fint_{B_r(x_0)\cap B_1'} |\partial_i\varphi(x')-\langle\partial_i\varphi\rangle_{x_0,r}|dx',\qquad \mathrm{with \ }\langle\partial_i\varphi\rangle_{x_0,r}=\fint_{B_r(x_0)\cap B_1'}\partial_i\varphi(x') \, dx',
\end{equation*}
is a Dini function for any $i=1,...,n-1$, i.e. satisfies \eqref{DINI}.

\subsection{$C^{1,\mathrm{DMO}}$ domains enjoy properties (P1) and (P2)}
As we remarked in \cite{DonJeoVit23}, after a standard local flattening of the $C^{1,\mathrm{DMO}}$ boundary, the validity of (P1) follows by $C^1$ boundary regularity up to a flat boundary where homogeneous Dirichlet or Neumann boundary conditions are prescribed, for solutions of PDEs with DMO coefficients. We refer to \cite[Proposition 2.7]{DonEscKim18} and \cite[Theorem 1.2]{DonLeeKim20}. Moreover, always after a standard flattening, property (P2) follows by the Hopf Lemma proved in \cite{RenSirSoa23} on flat boundaries and DMO coefficients.
\begin{remark}\label{r:DMO}
We would like to remark here that our result in $C^{1,\mathrm{DMO}}$ domains is not in contradiction with the counterexample in \cite{KenZha23}. In fact, the two dimensional example there is given by a local paramentrization with fails to be $C^{1,\mathrm{Dini}}$ but is convex. As we pointed out in \cite[Proposition 3.1]{DonJeoVit23}, a local $C^{1,\mathrm{DMO}}$ parametrization $\varphi$ which is convex satisfies the $C^{1,\mathrm{Dini}}$-paraboloid condition \cite{ApuNaz19}; that is,
\begin{equation*}
\omega(r)=\sup_{|x|\leq r}\frac{\varphi(x)}{|x|}
\end{equation*}
is a Dini function. The same consideration above explains also why the validity of the Hopf lemma in $C^{1,\mathrm{DMO}}$ domains is not in contradiction with the counterexample in \cite{ApuNaz16}.
\end{remark}

\subsection{Positive harmonic function vanishing on $\partial\Omega$ with no critical points}
The first step is the construction of a solution to \eqref{eqv}. The existence of such a function is done as in Section \ref{s:distance}. However, properties (P1)-(P2) together imply that $v\in C^1$ up to the boundary and that $\partial_\nu v<0$ on $\partial\Omega\cap B_1$, which says, that $v$ has no critical points in $\overline\Omega\cap B_1$.

\subsection{The harmonic conjugate}
Let us construct the harmonic conjugate $\overline v$ of $v$ as in Section \ref{s:conjugate}. Let us remark that $\overline v$ solves
\begin{equation*}
\begin{cases}
\Delta\overline v=0 &\mathrm{in} \ \Omega\cap B_1\\
\partial_\nu \overline v=0 &\mathrm{on} \ \partial\Omega\cap B_1.
\end{cases}
\end{equation*}
Then, by (P1) again, one has $\overline v\in C^1$ up to $\partial\Omega$.

\subsection{A hodograph conformal transformation and its invertibility}
Let us define the hodograph conformal mapping involving $v,\overline v$ as in Section \ref{s:hodograph}, i.e. $\Theta(x,y)=(\overline v(x,y),v(x,y))=(X,Y)$, which is of class $C^{1}$ with $\Theta(0)=0$. Then, since $v$ has no critical points in $\overline\Omega\cap B_1$, the invertibility of the map is even more direct this time, since $|\mathrm{det}D\Theta|>0$ in $\overline\Omega\cap B_1$. Hence, this time $\Theta$ is a $C^1$ local diffeomorphism between $\overline{E_r}$ and $\Theta(\overline{E_r})$. 

\subsection{Size control of the critical set}
Let us consider $U=u\circ\Theta^{-1}$, and consider again the odd reflection across $\{Y=0\}$. One ends up again with a harmonic function in the reflected set $\mathcal U$ for which $\Theta(\partial\Omega)$ is an interior nodal line $\{Y=0\}$. Then, since
\begin{equation*}
|\nabla u|^2=|\mathrm{det} D\Theta| \cdot |\nabla U|^2\circ\Theta,
\end{equation*}
we have
\begin{equation}\label{boundAlm}
\mathcal H^0(C(u)\cap \overline{E_r})\leq \mathcal H^0(C(\Theta)\cap \overline{E_r})+\mathcal H^0(C(U)\cap \mathcal U)<\infty.
\end{equation}
The latter is true due to classic size estimates of interior critical sets for harmonic functions.

\section*{Acknowledgement}
The author would like to thank Xavier Tolsa, Matteo Levi and Susanna Terracini for fruitful conversations on harmonic measure in NTA domains, and Giovanni Alessandrini for a nice correspondence which helped to lower the requirements of Theorem 1.2. The author is a research fellow of Istituto Nazionale di Alta Matematica INDAM group GNAMPA, supported by the GNAMPA projects E5324001950001 PDE ellittiche che degenerano su variet\'a di dimensione bassa e frontiere libere molto sottili, and E53C25002010001 Struttura fine e regolarit\'a in problemi variazionali non-lineari.

\end{document}